\documentclass[11pt,a4paper]{article}
\usepackage{amsmath,amssymb}
\usepackage{bm}
\usepackage{ascmac}
\usepackage[dvipdfmx]{graphicx}
\usepackage{enumerate}
\usepackage{tikz}
\usepackage[left=3cm, right=3cm]{geometry}
\usepackage{comment}

\newcommand{\rC}{\overrightarrow{C}}
\newcommand{\lC}{\overleftarrow{C}}

\usepackage{amsthm}
\newtheorem{thm}{Theorem}
\newtheorem{prethm}{Theorem}

\newtheorem{lem}[thm]{Lemma}
\newtheorem{con}[thm]{Conjecture}

\newtheorem{cla}{Claim}[section]
\newtheorem*{subcla4.6.1}{Subclaim 4.6.1}

\theoremstyle{definition}

\newtheorem{prob}[thm]{Problem}

\newtheorem*{prf}{Proof}
\newtheorem*{mainprf}{Proof of Theorem \ref{thm:main1}}

\usepackage{latexsym}
\def\qed{\hfill $\Box$}

\definecolor{mygray}{gray}{0.9}

\begin{document}

\title{On hamiltonian cycles of 1-tough $(P_2 \cup kP_1)$-free graphs} 

\author{%
    Masahiro Sanka\thanks{E-mail address: \texttt{sankamasa@gmail.com}}\\ 
    Department of Mathematics, Faculty of Science Division II, \\ Tokyo University of Science, \\
    1-3 Kagurazaka, Shinjuku, Tokyo 162-8601, Japan
} 

\date{}

\maketitle

\begin{abstract}
    Let $k$ be a positive integer.
    A graph is said to be $(P_2 \cup kP_1)$-free if it does not contain $P_2 \cup kP_1$ as an induced subgraph.
    Recently, Ota and the author asked whether every 1-tough and $k$-connected $(P_2 \cup kP_1)$-free graph is hamiltonian or the Petersen graph.
    Note that this problem is affirmative for $k \in \{1,2,3\}$ by the known results.
    In this paper, we show that for each integer $k \geq 4$, if $G$ is a $1$-tough and $(k-1)$-connected $(P_2 \cup kP_1)$-free graph with $|V(G)| \ge k^2+k+1$ and $\delta(G) \ge k$, then $G$ is hamiltonian.
    This result implies that the above question is affirmative for large graphs.
\end{abstract}


\section{Introduction}\label{sec:intro}
All graphs considered in this paper are finite and simple graphs.
We use the terminology not defined here in \cite{Diestel-2017}, and additional definitions will be given as needed.

A \emph{hamiltonian cycle} in a graph is a cycle passing through all the vertices of this graph.
We say that a graph is \emph{hamiltonian} if it has a hamiltonian cycle.
The notion of toughness was introduced in the study of hamiltonian cycles by Chv\'{a}tal \cite{Chvatal-1973}.
The \emph{toughness} of a graph $G$, denoted by $\tau(G)$, is defined by
\[\tau(G) = \min \left\{\frac{|S|}{\omega(G-S)} \mid S \subset V(G) \text{ and } \omega(G-S) \ge2 \right\},\]
or $\tau(G)=\infty$ if $G$ is a complete graph.
For a real number $t \ge 0$, we say that $G$ is \emph{$t$-tough} if $\tau(G) \ge t$.
Note that a graph $G$ is $t$-tough if and only if $|S| \ge t \cdot \omega(G-S)$ for any subset $S \subset V(G)$ with $\omega(G-S) \ge 2$.

It is well known that every hamiltonian graph is 1-tough, but the converse does not hold.
In 1973, Chv\'{a}tal conjectured that the converse holds at least in an approximate sense.

\begin{con}[\cite{Chvatal-1973}]\label{con:Chvatal}
    There exists a constant $t_0 >0$ such that every $t_0$-tough graph on at least three vertices is hamiltonian.
\end{con}
 
Bauer, Broersma and Veldman \cite{Bauer-Broersma-Veldman-2000} showed that for any $t<\frac{9}{4}$, there exists a non-hamiltonian $t$-tough graph.
Thus, if Conjecture \ref{con:Chvatal} is true, then such a constant $t_0$ must be at least $\frac{9}{4}$.
For details of studies on Conjecture \ref{con:Chvatal}, we refer the reader to the survey \cite{Bauer-Broersma-Schmeichel-2006}.

For two graphs $R_1$ and $R_2$, we define $R_1 \cup R_2$ as the disjoint union of them.
Also, for a graph $R$ and a positive integer $m$, we denote the disjoint union of $m$ copies of $R$ by $mR$.
For a graph $R$, we say that a graph $G$ is \emph{$R$-free} if it does not contain $R$ as an induced subgraph.

Conjecture \ref{con:Chvatal} has been verified in various classes of graphs, such as chordal graphs 
\cite{Chen-Jacobson-Kezdy-Lehel-1998, Kabela-Kaiser-2017}, $2K_2$-free graphs \cite{Broersma-Patel-Pyatkin-2014, Ota-Sanka-2022, Shan-2019}, claw-free graphs \cite{Kaiser-Vrana-2012}, $(P_2 \cup P_3)$-free graphs \cite{Shan-2021}, and so on.
In this paper we consider the hamiltonicity of $(P_2 \cup kP_1)$-free graphs for each positive integer $k$.
Several results are known for this class of graphs, involving strengthened toughness conditions 
\cite{Hatfield-Grimm-2021, Hu-Wang-Shen-2025, Li-Broersma-Zhang-2016, Nikoghosyan-2013, Ota-Sanka-2024, Shi-Shan-2022, Xu-Li-Zhou}.
In particular, Ota and the author proved the following theorem.

\begin{prethm}[Ota and Sanka \cite{Ota-Sanka-2024}]\label{pre:Ota-Sanka}
    Let $k \ge 2$ be an integer and $G$ be a $1$-tough and $k$-connected $(P_2 \cup kP_1)$-free graph.
    If $\delta(G) \ge \frac{3k-3}{2}$, then $G$ is hamiltonian or isomorphic to the Petersen graph.
\end{prethm}

In 2025, Hu, Wang and Shen extended the former result using the Fan-type condition.
As a corollary, they proved the following theorem.

\begin{prethm}[Hu, Wang and Shen \cite{Hu-Wang-Shen-2025}]\label{pre:Hu-Wang-Shen}
    Let $k \ge 2$ be an integer and $G$ be a $1$-tough and $k$-connected $(P_2 \cup kP_1)$-free graph.
    If $\delta(G) \ge \frac{7k-6}{5}$, then $G$ is hamiltonian or isomorphic to the Petersen graph.
\end{prethm}

We discuss the sharpness of the connectivity and the minimum degree conditions in Theorems \ref{pre:Hu-Wang-Shen}.
For $k=2$, Both Theorems \ref{pre:Ota-Sanka} and \ref{pre:Hu-Wang-Shen} imply that every $1$-tough and $(P_2 \cup 2P_1)$-free graph is hamiltonian, which is shown in \cite{Li-Broersma-Zhang-2016}.
For $k=3$, Both Theorems \ref{pre:Ota-Sanka} and \ref{pre:Hu-Wang-Shen} imply that every $1$-tough and 3-connected $(P_2 \cup 3P_1)$-free graph is hamiltonian or isomorphic to the Petersen graph.
It is also known that there are infinitely many $1$-tough and 2-connected (but not 3-connected) $(P_2 \cup 3P_1)$-free graph which are not hamiltonian.
On the other hand, for $k \ge 4$, the sharpness of the connectivity and the minimum degree conditions in Theorem \ref{pre:Hu-Wang-Shen} remains unknown.
It has been noted that for each $k \ge 4$, there exist infinitely many $1$-tough and $(k-2)$-connected $(P_2 \cup kP_1)$-free graphs that are not hamiltonian.
In \cite{Ota-Sanka-2024}, Ota and the author states that the following problem might be affirmative, which would constitute an extension of the classical theorem due to Chv\'{a}tal and Erd\H{o}s \cite{Chvatal-Erdos-1972}.

\begin{prob}[Ota and Sanka \cite{Ota-Sanka-2024}]\label{prob:Ota-Sanka}
    For $k \ge 4$, is every $1$-tough and $k$-connected $(P_2 \cup kP_1)$-free graph hamiltonian?
\end{prob}

In this paper, we show the following theorem, which demonstrates that Problem~\ref{prob:Ota-Sanka} is affirmative for large graphs.

\begin{thm}\label{thm:main1}
    Let $k \geq 4$ be an integer, and let $G$ be a $1$-tough and $(k-1)$-connected $(P_2 \cup kP_1)$-free graph.
    If $|V(G)| \ge k^2+k+1$ and $\delta(G) \ge k$, then $G$ is Hamiltonian.
\end{thm}

In Theorem \ref{thm:main1}, the condition that the graph is $k$-connected is improved to a stronger condition that the graph is $(k-1)$-connected with minimum degree at least $k$.
For such improvement, we will use the next theorem in the proof of Theorem \ref{thm:main1}.
For a graph $G$ and a cycle $C$ of $G$, we say that $C$ is an \emph{edge-dominating cycle} if every component of $G-V(C)$ consists of only one vertex.

\begin{thm}\label{thm:main3}
    Let $k \geq 4$ be an integer, and let $G$ be a $1$-tough and $(k-1)$-connected $(P_2 \cup kP_1)$-free graph.
    Then, every longest cycle of $G$ is an edge-dominating cycle of $G$.
\end{thm}
 
Theorem \ref{thm:main3} can be proven by using the proof method of the following theorem due to Bigalke and Jung \cite{Bigalke-Jung-1979}.

\begin{prethm}[Bigalke and Jung \cite{Bigalke-Jung-1979}]\label{pre:Bigalke-Jung}
    Let $G$ be a $1$-tough graph with $\kappa(G) \ge 3$.
    If $\alpha(G) \le \kappa(G)+1$, then $G$ is hamiltonian or isomorphic to the Petersen graph.
\end{prethm}

In the proof of Theorem \ref{pre:Bigalke-Jung}, we consider a longest cycle $C$ in the given graph $G$.
If $C$ is not a hamiltonian cycle, then there is a component $H$ of $G-V(C)$.
Then, if $G$ is not isomorphic to the Petersen graph, we can find an independent set $X \subset V(C) \setminus N_G(H)$ with $|X| \ge \kappa(G)+1$. 
However, then for a vertex $h \in V(H)$, the set $X \cup \{h\}$ is an independent set, a contradiction to $\alpha(G) \le \kappa(G)+1$.

The proof of Theorem \ref{thm:main3}, which we only give its outline here, is quite similar to the proof of Theorem \ref{pre:Bigalke-Jung}.
For a 1-tough and $(k-1)$-connected $(P_2 \cup kP_1)$-free graph $G$, suppose that there is a longest cycle $C$ which is not an edge-dominating cycle.
Then, there is a component $H$ of $G-V(C)$ with $|V(H)| \ge 2$.
By considering the same way as the proof of Theorem \ref{pre:Bigalke-Jung}, we can find an independent set $X \subset V(C) \setminus N_G(H)$ with $|X| \ge k$.
However, then for an edge $h_1h_2 \in E(H)$, $G[X \cup \{h_1,h_2\}]$ is isomorphic to $P_2 \cup kP_1$, a contradiction.
In this way, we obtain Theorem \ref{thm:main3}.

The rest of this paper is organized as follows.
In Section \ref{sec:pre}, we prepare several terms and lemmas.
In Section \ref{sec:result1}, we prove Theorem \ref{thm:main1}.
Finally, in Section \ref{sec:final}, we conclude the paper and discuss open problems.

\section{Preliminary}\label{sec:pre}

In this section, we introduce terminology and lemmas.

For two integers $a,b$ with $a \le b$, we let $[a,b]=\{x \in \mathbb{Z} \mid a \le x \le b\}$.
Let $G$ be a graph and let $v \in V(G)$.
The set of neighbors of $v$ in $G$ is denoted by $N_G(v)$ and the degree of $v$ in $G$ is denoted by $d_G(v)$.
Let $S \subset V(G)$.
We define the neighbors of $S$ by $N_G(S)=\bigcup_{x \in S}(N_G (x) \setminus S)$.
Let $S_1,S_2 \subset V(G)$ with $S_1 \cap S_2 = \emptyset$.
Then, the number of edges between $S_1$ and $S_2$ in $G$ is denoted by $e_G(S_1,S_2)$.
For $S \subset V(G)$, $G[S]$ denotes the subgraph of $G$ induced by $S$.

Next, we give some notation for paths and cycles.
Let $P$ be a path.
For $x,y \in V(P)$, we denote by $xPy$ the path from $x$ to $y$ passing through $P$.
Let $C$ be a cycle with a given orientation.
For $x \in V(C)$, we denote by $x^{+}$ and $x^{-}$ the successor and the predecessor of $x$, respectively.
For $u,v \in V(C)$, we denote by $u \overrightarrow{C}v$ the $uv$-path from $u$ to $v$ along the orientation of $C$.
Also, $u \overleftarrow{C}v$ denotes $v \overrightarrow{C} u$.
For a subset $I \subset V(C)$, let $I^+=\{x^+ \mid x \in I\}$ and $I^-=\{x^- \mid x \in I\}$.

Finally, we introduce the following lemma on $(P_2 \cup kP_1)$-free graphs for each positive integer $k$, which will be used in Section \ref{sec:result1}.
\begin{lem}[{\cite[Lemma 2.3]{Ota-Sanka-2024}}]\label{lem:P2kP1}
    Let $k \ge 1$ be an integer, $G$ be a $(P_2 \cup kP_1)$-free graph and $X \subset V(G)$ be an independent set of $G$.
    Then, for each vertex $v \in V(G)$, $|N_G(v) \cap X| \ge |X|-k+1$ if $N_G(v) \cap X \neq \emptyset$.
\end{lem}

\section{Proof of Theorem \ref{thm:main1}}\label{sec:result1}

In this section, we prove Theorem \ref{thm:main1}.

\begin{mainprf}
Let $k \ge 4$ be an integer and $G$ be a 1-tough and $(k-1)$-connected $(P_2 \cup kP_1)$-free graph with $\delta(G) \ge k$ and $|V(G)| \ge k^2+k+1$.
Assume to the contrary that $G$ is not hamiltonian.
Let $C$ be a longest cycle of $G$.
Then $V(G) \setminus V(C) \neq \emptyset$ as $G$ is not hamiltonian.

We choose the longest cycle $C$ and a vertex $h \in V(G) \setminus V(C)$ so that $d_G(h)$ is as large as possible.
By Theorem \ref{thm:main3}, $C$ is an edge-dominaing cycle.
Thus, we have $N_G(h) \subset V(C)$.

We fix an orientation of $C$.
Let $d=|N_G(h)|$.
Then, since $\delta(G) \ge k$, we have $d \ge k$.
Let $v_1,v_2, \ldots, v_d$ be the vertices in $N_G(h)$ appearing in this order along the orientation of $C$.
For each $i \in [1,d]$, let $u_i=v_i^+$ and $w_i=v_{i+1}^-$, where the indices are taken modulo $d$.
Let $U=\{u_i \mid i \in [1,d]\}=N_G(H)^+$.

\begin{cla}\label{cla:longest cycle}
    The set $U \cup \{h\}$ is independent in $G$.
\end{cla}

Claim \ref{cla:longest cycle} can be shown easily by the fact that $C$ is a longest cycle of $G$.
So we omit the proof.

\begin{cla}\label{cla3.1}
The set $V(G) \setminus N_G(U)$ is independent in $G$.
Thus, we have
\begin{equation}
    |N_G(U)| \ge \frac{|V(G)|}{2}. \label{eq3-1}
\end{equation}
Moreover, we have $N_G(U) \subset V(C)$.
\end{cla}
\begin{prf}
    If $xy \in E(G)$ for some $x,y \in V(G) \setminus N_G(U)$, then $G[\{x,y,u_1,\ldots,u_k\}]$ is isomorphic to $P_2 \cup kP_1$, a contradiction.
    Thus, the set $V(G) \setminus N_G(U)$ is independent in $G$.
    Since $G$ is 1-tough, it holds that $|V(G)|-|N_G(U)|=|V(G) \setminus N_G(U)|  \le \alpha(G) \le \frac{|V(G)|}{2}$, which implies (\ref{eq3-1}).

    For $y \in V(G) \setminus V(C)$, suppose that $yu_i \in E(G)$ for some $u_i$.
    Then, since $C$ is a longest cycle of $G$, we have $y \neq h$ and $N_G(y) \cap U=\{u_i\}$.
    However, then for a subset $U' \subset U \setminus \{u_i\}$ with $|U'|=k-1$, the graph $G[\{y,u_i,h\} \cup U']$ is an induced subgraph of $G$ isomorphic to $P_2 \cup kP_1$, a contradiction.
    Thus, for any $y \in V(G) \setminus V(C)$, we have $y \notin N_G(U)$, that is, $y \in V(G) \setminus N_G(U)$.
    Therefore, we have $V(G) \setminus V(C) \subset V(G) \setminus N_G(U)$, equivalently $N_G(U) \subset V(C)$.
    \qed
\end{prf}

\begin{cla}\label{cla3.2}
    Let $y \in N_G(U)$.
    Then, $|N_G(y) \cap U| \ge d-k+1$.
    Also, if $y \notin N_G(h)$, then $|N_G(y) \cap U| \ge d-k+2$.
    This together with the inequality (\ref{eq3-1}) implies
    \begin{equation}
    e_G(U, N_G(U)) \ge (d-k+2)|N_G(U)|-d \ge \frac{d-k+2}{2}|V(G)|-d.\label{eq3-2}
\end{equation}
\end{cla}
\begin{prf}
    Let $y \in N_G(U)$.
    Then, Lemma \ref{lem:P2kP1} implies $|N_G(y) \cap U| \ge |U|-k+1=d-k+1$.
    If $y \notin N_G(h)$, then Lemma \ref{lem:P2kP1} implies $|N_G(y) \cap U| = |N_G(y) \cap (U \cup \{h\})| \ge d-k+2$.\qed
\end{prf}

\begin{cla}\label{cla3.3}
    For each $u_i \in U$, it holds that $d_G(u_i) \le d$.
    This implies
    \begin{equation}
    e_G(U, N_G(U)) \le d^2.\label{eq3-3}
    \end{equation}
\end{cla}
\begin{prf}
    Suppose that $d_G(u_i) > d$ for some $u_i \in U$.
    If $u_i^+ \in N_G(h)$, that is, $u_i^+=v_{i+1}$, then 
    \[C'=v_i \overleftarrow{C} v_{i+1} h v_i \]
    is a longest cycle of $G$ with $u_i \in V(G) \setminus V(C')$.
    However, the choice of $C$ and $h$ implies $d=d_G(g) \ge d_G(u_i)>d$, a contradiction.
    Hence $u_i^+ \notin N_G(h)$.
    Then, since $u_i^+ \in N_G(u_i) \setminus N_G(h) \subset N_G(U) \setminus N_G(h)$, we have $|N_G(u_i^+) \cap U| \ge d-k+2 \ge 2$ by Lemma~\ref{lem:P2kP1}.
    Thus, there exists $u_j \in U \setminus \{u_i\}$ such that $u_j u_i^+ \in E(G)$.
    Let 
    \[C''=v_i \overleftarrow{C} u_j u_i^+ \overrightarrow{C} v_j h v_i.\]
    Then $C''$ is also a longest cycle of $G$ and $u_i \in V(G) \setminus V(C'')$.
    However, the choice of $C$ and $h$ also implies $d=d_G(h) \ge d_G(u_i)>d$, a contradiction.
    Therefore, $d_G(u_i) \le d$ for every $u_i \in U$. \qed    
\end{prf}

\begin{cla}\label{cla:3.5}
    It holds that $d >\frac{k^2-k-2}{2}$.
\end{cla}
\begin{prf}
    Assume to the contrary that $d \le \frac{k^2-k-2}{2}$. Note that $k < \frac{k^2-k-2}{2}$ since $k \ge 4$.
    By the inequalities (\ref{eq3-2}) and (\ref{eq3-3}), we have $|V(G)| \le \frac{2d(d+1)}{d-k+2}$.
    Then, since the function $f(d)=\frac{2d(d+1)}{d-k+2}$ attains its maximum value $f(k)=k^2+k$ on $k \le d \le \frac{k^2-k-2}{2}$, we have $|V(G)| \le k^2+k$, a contradiction to $|V(G)| \ge k^2+k+1$. \qed
\end{prf}

\begin{cla}\label{cla:3.6}
    There is a vertex $x \in V(C)$ such that $\{x,x^+\} \subset N_G(U)$.
\end{cla}
\begin{prf}
    Note that for any edge $xy \in E(C)$, we have $\{x,y\} \cap N_G(U) \neq \emptyset$ by Claim \ref{cla3.1}.
    Suppose to the contrary that there is no vertex $x \in V(C)$ such that $\{x,x^+\} \subset N_G(U)$.
    Then, the vertices on $C$ are alternatingly contained in $N_G(U)$ and $V(C) \setminus N_G(U)$.
    However, then Claim \ref{cla3.1} implies $|V(G) \setminus N_G(U)|=|V(G)|-|V(C)|+\frac{|V(C)|}{2}>\frac{|V(G)|}{2}$, a contradiction to $|V(G) \setminus N_G(U)| \le \frac{|V(G)|}{2}$ since $G$ is 1-tough. \qed
\end{prf}

We fix a vertex $x \in V(C)$ with $x,x^+ \in N_G(U)$.
Note that $x \notin N_G(h)$ as $v_i^+=u_i \notin N_G(U)$ for any $v_i \in N_G(h)$.
Therefore, we may assume $x \in V(u_d^+ \rC w_d)$.
We define $l_x$ and $r_x$ as 
\[l_x= \min\{i \in [1,d] \mid u_i x \in E(G)\} \text{ and } r_x=\max\{i \in [1,d] \mid u_i x^+ \in E(G)\}.\]

\begin{cla}\label{cla3.7}
    The following statements hold.
    \begin{enumerate}[(i)]
        \item $N_G(x) \cap U \subset \{u_i \mid i \in [l_x,d]\}$ and $|N_G(x) \cap U| \ge d-k+2$,
        \item $N_G(x^+) \cap U \subset \{u_i \mid i \in [1,r_x]\}$ and $|N_G(x^+) \cap U| \ge d-k+1$, and
        \item $u_{l_x}^+ \notin N_G(h)$, $N_G(u_{l_x}^+) \cap U \subset \{u_i \mid i \in [r_x,l_x]\}$ and $|N_G(u_{l_x}^+) \cap U| \ge d-k+2$.
    \end{enumerate}
\end{cla}
\begin{prf}
    By the definitions of $l_x$ and $r_x$, we have $N_G(x) \cap U \subset \{u_i \mid i \in [l_x,d]\}$ and $N_G(x^+) \cap U \subset \{u_i \mid i \in [1,r_x]\}$.
    Since $x \notin N_G(h)$, Claim \ref{cla3.2} implies $|N_G(x) \cap U| \ge d-k+2$ and $|N_G(x^+) \cap U| \ge d-k+1$.
    Therefore, we obtain (i) and (ii).

    We consider the vertex $u_{l_x}^+$.
    If $l_x < r_x$, then
    \[xu_{l_x} \overrightarrow{C} v_{r_x} h v_{l_x} \overleftarrow{C} x^+ u_{r_x} \overrightarrow{C}x\]
    is a cycle longer than $C$ in $G$, a contradiction.
    Thus, $l_x \ge r_x$.
    If $u_{l_x}^+ \in N_G(h)$, then
    \[h v_{r_x} \lC x^+ u_{r_x} \rC u_{l_x} x \lC u_{l_x}^+ h\]
    is a cycle longer than $C$ in $G$, a contradiction.
    Thus, $u_{l_x}^+ \notin N_G(h)$.
    As $u_{l_x}^+ \in N_G(u_{l_x}) \subset N_G(U)$, Claim \ref{cla3.2} implies $|N_G(u_{l_x}^+) \cap U| \ge d-k+2 \ge 2$.
    If $u_{l_x}^+ u_j \in E(G)$ for some $j \in [1,r_x-1]$, then
    \[h v_j \lC x^+ u_{r_x} \rC u_{l_x} x \lC u_{l_x}^+ u_j \rC v_{r_x} h\]
    is a cycle in $G$ longer than $C$, a contradiction.
    If $u_{l_x}^+ u_j \in E(G)$ for some $j \in [l_x+1,d]$, then
    $$h v_{r_x} \lC x^+ u_{r_x} \rC u_{l_x} x \lC u_j u_{l_x}^+ \rC v_j h$$
    is a cycle in $G$ longer than $C$, a contradiction.
    Therefore, we obtain $N_G(u_{l_x}^+) \cap U \subset \{u_i \mid i \in [r_x,l_x]\}$, and thus (iii) holds. \qed
\end{prf}

    The statements (i), (ii) and (iii) in Claim \ref{cla3.7} imply that
        \begin{eqnarray*}
            l_x &\le& k-1,\label{eq4-1}\\
            r_x &\ge& d-k+1, \text{ and} \label{eq4-2}\\
            d-k+1 \le l_x-r_x &\le& 2k-d-2 . \label{eq4-3}
        \end{eqnarray*}
    Thus, we get $d \le \frac{3k-3}{2}$.
    This together with Claim \ref{cla:3.5}, we obtain  $\frac{k^2-k-2}{2} <d \le \frac{3k-3}{2}$.
    However, then $k < 2+\sqrt{3}<4$, a contradiction to $k \ge 4$.
This completes the proof of Theorem~\ref{thm:main1}. \qed
\end{mainprf}

\section{Remaining problems}\label{sec:final}

In this paper, we have shown that for an integer $k \ge 4$ and a 1-tough and $(k-1)$-connected $(P_2 \cup kP_1)$-free graph $G$, if $\delta(G) \ge k$ and $|V(G)| \ge k^2+k+1$, then $G$ is hamiltonian.
We believe that Problem \ref{prob:Ota-Sanka} is affirmative.
Actually, for $k \in \{4,5\}$, we have succeeded to prove that every $1$-tough and $(k-1)$-connected $(P_2 \cup kP_1)$-free graph with minimum degree at least $k$ is hamiltonian.
(We are preparing a paper on it now.)

\begin{figure}[htbp]
    \centering
    \includegraphics{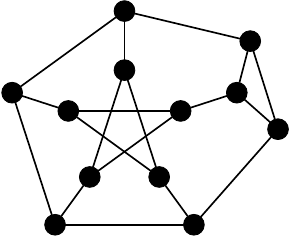}

    \caption{A non-hamiltonian $1$-tough and $3$-connected $(P_2 \cup 4P_1)$-free graph}\label{fig6:non hamiltonian P24P1-free graph}
\end{figure}

Finally, we make some remarks on the minimum degree condition.
Based on Theorems~\ref{thm:main1} and \ref{pre:Bigalke-Jung}, there is a point of interest regarding the hamiltonicity of 1-tough and $(k-1)$-connected $(P_2 \cup kP_1)$-free graph $G$ with $\delta(G)=k-1$.
We have found a non-hamiltonian 1-tough and 3-connected $(P_2 \cup 4P_1)$-free graph; an example is shown in Figure~\ref{fig6:non hamiltonian P24P1-free graph}.
However, we do not find a large example. 
So we have the following problem.

\begin{prob}\label{newcon2}
    For each integer $k \geq 4$, is there a constant $c_k$ such that every $1$-tough and $(k-1)$-connected $(P_2 \cup kP_1)$-free graph on at least $c_k$ vertices is hamiltonian?
\end{prob}

\subsection*{Acknowledgement}

We would like to thank Professor Katsuhiro Ota and Professor Michitaka Furuya for their valuable comments.
This work was supported by the Research Institute for Mathematical Sciences, an International Joint Usage/Research Center located in Kyoto University.

\subsection*{Data availability}

No new data is created or analysed in this study.

\section*{Declarations}

\subsection*{Conflicts of interest}
The author has no relevant financial or non-financial interests to disclose.

\bibliography{Phd} 
\bibliographystyle{myplain} 

\end{document}